\documentclass[12pt]{article}

\usepackage{mathrsfs}
\usepackage{latexsym}
\usepackage{amsmath, amsthm, amssymb}
\input amssym.def
\input amssym

\def \as {association scheme}

\def \c {\Bbb C}

\def \hrow {\rule{0pt}{18pt}}
\def \hhrow {\rule{0pt}{24pt}}

\def \mc {\mathcal}
\def \mbb {\mathbb}

\def \xr {$(X, \{R_i\}_{i \in [d]})$}
\def \ys {$(Y, \{S_j\}_{j \in [e]})$}

\def \x {\mathcal{X}}
\def \y {\mathcal{Y}}
\def \t {\mathcal{T}}
\def \u {\mathcal{U}}
\def \i {\mathcal{I}}

\def \< {\langle }
\def \> {\rangle }

\newtheorem{thm}{Theorem}[section]

\newtheorem{lemma}[thm]{Lemma}
\newtheorem{prop}[thm]{Proposition}
\newtheorem{cor}[thm]{Corollary}

\title{On Wreath Products of One-Class Association Schemes}

\begin{document}

\author{Sung Y. Song$^{\rm{a}}$ and Bangteng Xu$^{\rm{b}}$\\
\small{$^{\rm{a}}$ Department of Mathematics, Iowa State University,
Ames, IA, 50011, U. S. A.}\\ \small{$\langle$sysong@iastate.edu$\rangle$}\\
\small{$^{\rm{b}}$ Department of Mathematics \& Statistics, Eastern
Kentucky University,}\\ \small{Richmond, KY, 40475, U. S. A.} \\
\small{$\langle$Bangteng.Xu@eku.edu$\rangle$}}

\date{}
\maketitle

\begin{abstract}
We give a full description of the algebraic structures of the
Bose-Mesner algebra and Terwilliger algebra of the wreath product of
one-class association schemes.

\medskip \noindent \textbf{Keywords}: Bose-Mesner algebra;
commutative association scheme; Terwilliger algebra; Wreath product.

\smallskip \noindent \textbf{Classification codes}:  05E30, 05C50,
15A78

\end{abstract}

\section{Introduction}
The wreath product in the theory of association schemes provides a
way to construct new (imprimitive) association schemes from old.
Recently the Terwilliger algebra of the wreath product of one-class
association schemes was described by G. Bhattacharyya, S. Y. Song
and R. Tanaka \cite{BS}. It was shown that all irreducible modules
except for the primary module of the algebra were one-dimensional.
There are not many association schemes which have the property:
``All non-primary irreducible modules of the Terwilliger algebra are
one-dimensional." In fact, it was proved by R. Tanaka \cite{Ta} that
the class of association schemes coming as the wreath product of
one-class association schemes and that of group schemes of finite
abelian groups are the only ones that hold this property.

In this paper we revisit the wreath product of one-class association
schemes to give a complete structural description of its Bose-Mesner
algebra and Terwilliger algebra. Our work is motivated by the work
of F. Levenstein, C. Maldonado and D. Penazzi \cite{LMP06} which
gives the description of the Terwilliger algebra of the Hamming
scheme $H(d, q)$ as symmetric $d$-tensors of the Terwilliger algebra
of the one-class association scheme $H(1, q)$ (or $K_q$). It is also
motivated by P. Terwilliger \cite{Te92} and E. Egge \cite{Eg00} in
the efforts of determination of an abstract version of the
Terwilliger algebra for a given association scheme.

\section{Preliminaries and main results}
In this section, we briefly recall some basic facts about the
Bose-Mesner algebra and Terwilliger algebra of an association
scheme, and the definition of the wreath product of association
schemes (cf. \cite{BI84, So02, Te92, Zi}). Then we introduce our
main findings on the structural properties of Bose-Mesner algebra
and Terwilliger algebra of the wreath product of one-class
association schemes. These findings are formulated as Theorem
\ref{thm-bose}, Theorem \ref{thm-ter}, and Theorem \ref{thm-ter-1}
below.

Let $v$ and $d$ be positive integers. Throughout the paper we will
use $[d]$ to denote the set $\{0, 1, 2,\dots, d\}$ of the first
$d+1$ whole numbers. Let $X$ denote an $v$-element set, and let
$M_X(\mbb{C})$ denote the $\mbb{C}$-algebra of matrices whose rows
and columns are indexed by $X$. Let $R_0,R_1,\ldots,R_d$ be nonempty
relations on $X$, and let $A_0, A_1,\dots, A_d$ be the adjacency
matrices of the relations defined by $(A_i)_{xy}=1$ if $(x,y) \in
R_i$; $0$ otherwise. The pair $\mathcal{X}=(X,\{R_{i}\}_{i\in [d]})$
is called a \textit{$d$-class $(symmetric)$ association scheme of
order $v$} if the following hold:
\begin{enumerate}
\item[(1)] $A_0= I$,
\item[(2)] $A_0 + A_1 + \cdots + A_d= J$,
\item[(3)] $A_i^{t}= A_i$ for all $i \in [d]$,
\item[(4)] $A_iA_j = \sum\limits_{h=0}^{d} p_{ij}^{h}A_h$ for some
nonnegative integers $p_{ij}^h$, for all $h,i,j\in [d]$,
\end{enumerate}
where $I=I_v$ and $J=J_v$ are the $v \times v$ identity matrix and
all-one matrix, respectively, and $A^{t}$ denotes the transpose of
$A$. Our association scheme is also \textit{commutative}: that is,
$A_iA_j=A_jA_i$ for all $i,j\in [d]$, since all $A_i$ are symmetric.

The numbers $p_{ij}^h$ are called the \textit{intersection numbers}
and satisfy
$$p_{ij}^{h} =|\{ z \in X :\ (x,z) \in R_{i}, (z,y) \in R_{j}\}|,$$
where $(x,y)\in R_h.$ Given an element $x \in X$, let $R_i(x)= \{y
\in X:\ (x,y) \in R_i\}$. Then the intersection number
$p^0_{ii}=|R_i(x)|$, and is called the $i$th-valency of $\mc{X}$.
The $i$th-valency is denoted $k_i$. It is convenient to represent a
given association scheme $\mc{X}$ with its adjacency matrices $A_0,
A_1, \dots, A_d$ by the matrix $R(\mc{X}):=\sum\limits_{i=0}^d
iA_i$. The matrix $R(\mc{X})$ is called the association
\textit{relation matrix} of $\mc{X}$. The $(d+1)$-dimensional
algebra $\mathcal{A}=\langle A_0 , A_1 , \dots , A_d \rangle$ is a
semi-simple algebra known as the \textit{Bose-Mesner algebra} of
$\mathcal{X}$. The algebra admits a second basis $E_0 , E_1 , \dots
, E_d$ of primitive idempotents.

Given $X$ and $M_X(\mbb{C})$, by the standard module of $X$, we mean
the $v$-dimensional vector space $V=\mbb{C}^X=\bigoplus\limits_{x
\in X}\mbb{C} \hat{x}$ of column vectors whose coordinates are
indexed by $X$. For each $x \in X$, we denote by $\hat{x}$ the
column vector with $1$ in the $x$th position, and $0$ elsewhere.
Observe that $M_X(\mbb{C})$ acts on $V$ by left multiplication. We
endow $V$ with the Hermitian inner product. For a given association
scheme $\mathcal{X}$, the vector space $V$ can be written as the
direct sum of $V_i=E_iV$ where $V_i$ are the maximal common
eigenspaces of $A_0 , A_1 , \dots , A_d$. Given an element $x \in
X$, let $V_{i}^{*}=V_{i}^{*}(x) =\bigoplus\limits_{y \in
R_i(x)}\mbb{C} \hat{y}$. Both $R_i(x)$ and $V_i^*$ are referred to
as the $i$th \textit{subconstituent} of $\mathcal{X}$ with respect
to $x$. Let $E_i^{*}=E_i^{*}(x)$ be the orthogonal projection map
from $V=\bigoplus\limits_{i=0}^{d}V_{i}^{*}$ to the $i$th
subconstituent $V_{i}^{*}$. So, $E_i^{*}$ can be represented by the
diagonal matrix given by
$$(E_i^{*})_{yy}=\left
\{\begin{array}{lc}1 &\mbox{if } (x,y)\in R_i\\ 0 & \mbox{if } (x,y)
\notin  R_i \end{array}\right ..$$ The matrices
$E_0^{*},E_1^{*},\dots,E_d^{*}$ form a basis for a subalgebra
$\mathcal{A}^{*}=\mathcal{A}^{*}(x) = \langle E_i^{*} \rangle_{i\in
[d]}$ of $M_X(\mbb{C})$. The algebra $\mathcal{A}^{*}$ is
a commutative, semi-simple subalgebra of $M_X(\mbb{C})$. This
algebra is called the dual Bose-Mesner algebra of $\mathcal{X}$ with
respect to $x$. Let $\mathcal{T}=\mathcal{T}(x)$ denote the
subalgebra of $M_X(\mbb{C})$ generated by the Bose-Mesner algebra
$\mc{A}$ and the dual Bose-Mesner algebra $\mc{A}^{*}$. We call
$\mathcal{T}$ the {\it Terwilliger algebra of $\mc{X}$ with respect
to $x$}.

Terwilliger observed the following relations between triple products
$E_{i}^{*}A_jE_{h}^{*}$ and the intersection numbers of the scheme.

\begin{prop}\cite[Lemma 3.2]{Te92} \label{tripro} For $h, i, j \in [d]$,
$E_{i}^{*}A_jE_{h}^{*}=0$ if and only if $p_{ij}^{h}= 0$.\end{prop}

The algebra $\mathcal{T}$ is generated by the set $\{E_i^*A_jE_h^*:\
i,j,h \in [d]\}$ of `triple products' as an algebra, but in general,
the set $\{E_i^*A_jE_h^*:\  i,j,h \in [d]\}$ spans a proper linear
subspace of $\mathcal{T}$. However, in \cite{Mu93}, A. Munemasa
described the combinatorial characteristics of the association
schemes for which the set of triple products spans the entire space
$\mathcal{T}$ as in the following.

\begin{prop} \cite{Mu93} The set $\{E_i^*A_jE_h^*:\  i,j,h \in [d]\}$
spans $\mc{T}$ for each $x\in X$ if and only if $\mathcal{X}$ is
{\it triply-regular}, that is, the size $p_{ijh}^{lmn}(y,z,w)$ of
the set $R_i(y)\cap R_j(z)\cap R_h(w)$ is a constant $p_{ijh}^{lmn}$
for all triples $y,z,w\in X$ with $(y,z) \in R_l$, $(y,w) \in R_m$
and $(z,w) \in R_n$.\end{prop}

The wreath product of one-class association schemes that we study in
this paper is triply-regular (cf. \cite{Bh, BS}); and so, the set of
triple products spans its Terwilliger algebra with respect to each
element of the wreath product.

\medskip

We now recall the notion of the wreath product of two association
schemes. Let $\mathcal{X}=(X, \{R_i\}_{i\in[d]})$ and
$\mathcal{Y}=(Y, \{S_j\}_{j\in[e]})$ be association schemes of order
$|X|= v$ and $|Y|= u$. The \textit{wreath product} $\mathcal{X} \wr
\mathcal{Y}= (X \times Y, \{W_l\}_{l\in[d+e]})$ of $\mathcal{X}$ and
$\mathcal{Y}$ is a $(d+e)$-class association scheme, defined by:
\begin{itemize}
\item[] $W_0 = \{ ((x,y),(x,y)) :\ (x,y)\in X \times Y\}$;
\item[] $W_l = \{ ((x_1,y),(x_2,y)) :\ (x_1,x_2)\in R_l , y
    \in Y\}$ for $1 \leq l \leq d$ ; and
\item[] $W_l = \{ ((x_1,y_1),(x_2,y_2)) :\ x_1,x_2 \in X,
    (y_1,y_2) \in S_{l-d}\}$ for $d+1 \leq l \leq d+e$.
\end{itemize}
The association relation matrix  of $\mc{X}\wr\mc{Y}$ is described,
in terms of $R(\mc{X})$ and $R(\mc{Y})$, by
 $$ R(\mathcal{X} \wr \mathcal{Y})=\sum\limits_{l=0}^{d+e}lW_l= I_u \otimes R(\mathcal{X}) + \{R(\mathcal{Y})+
 d(J_u-I_u)\}\otimes J_v.$$
Let $A_0,A_1,\dots, A_d$ and $B_0,B_1,\dots, B_e$ be the adjacency
matrices of $\mathcal{X}$ and those of $\mathcal{Y}$, respectively.
Then the adjacency matrices $C_i$ of $\mathcal{X} \wr \mathcal{Y}$
are given by
$$C_0=B_0 \otimes A_0,\ C_1=B_0\otimes A_1,\dots,\ C_{d}=B_0\otimes
A_d,\ C_{d+1}=B_1\otimes J_v, \dots,\ C_{d+e}=B_e\otimes J_v,$$
where $\otimes$ denotes the Kronecker product: $A\otimes B =
(a_{ij}B)$ of two matrices $A=(a_{ij})$ and $B$. Here and in what
follows, we refer to the tensor product $A\otimes B$ of $A \in
M_X(\mbb{C})$ and $B\in M_Y(\mbb{C})$ as the Kronecker
product of $A$ and $B$ in $M_{X\times Y}(\mbb{C})$. 

In order to investigate the algebraic structure of the Bose-Mesner
algebra and Terwilliger algebra of the wreath product of one-class
association schemes, we shall need the notion of the \textit{complex
product} and \textit{subschemes} by following \cite{Zi}. Let $\x =$
\xr\ be an \as, and let $A_0, A_1, ..., A_d$ be the adjacency
matrices. For any relations $R_i$ and $R_j$, define
$$
R_i R_j := \{ R_h :\ p_{ij}^h \ne 0 \}.
$$
Let $\Delta$ be a nonempty subset of $[d]$. Then $\{R_i\}_{i \in
\Delta}$ is called a {\it closed subset} if $R_h R_j \subseteq
\{R_i\}_{i \in \Delta}$ for any $h, j \in \Delta$. If $\{R_i\}_{i\in
\Delta}$ is a closed subset, then the $\c$-space with basis
$\{A_i\}_{i \in \Delta}$ is a subalgebra of the Bose-Mesner algebra
of $\x$, called a {\it Bose-Mesner subalgebra} of $\x$, and denoted
by $\< A_i \> _{i \in \Delta}$. In this case, for each $x\in X$, let
$R_{\Delta}(x):=\bigcup_{i\in\Delta}R_i(x)$. With the given closed
subset $\{R_i\}_{i\in\Delta}$ and $X_\Delta=R_{\Delta}(x)$ for an
arbitrarily fixed $x\in X$, the pair $(X_\Delta,
\{\overline{R_i}\}_{i\in \Delta})$ where
$\overline{R_i}:=R_i\cap (X_\Delta\times X_\Delta)$ forms an
association scheme \cite{Zi}. This scheme is called the
\textit{subscheme} of $\mc{X}$ with respect to $x$ and closed subset
$\{R_i\}_{i\in\Delta}$, and is denoted by $\mc{X}_\Delta$.
Note that the cardinality of the set $X_\Delta$ is
$\sum_{i \in \Delta} k_i$, and $\sum_{i \in \Delta} k_i$ divides
$\sum_{i \in [d]} k_i$. Furthermore, it is
shown that the Bose-Mesner algebra of $\mc{X}_\Delta$ is `exactly
isomorphic' to $\< A_i \> _{i \in \Delta}$ in the following sense.

Let $\x =$ \xr\ be an \as, and let $A_0, A_1, ..., A_d$ be the
adjacency matrices of $\mc{X}$. Let $\y = $ \ys\ be an \as, and let
$B_0, B_1, \dots, B_e$ be the adjacency matrices of $\mc{Y}$. Let
$\{R_i\}_{i \in \Delta}$ be a closed subset of $\x$, and $\{S_j\}_{j
\in \Lambda}$ a closed subset of $\y$. We say that the Bose-Mesner
subalgebras $\< A_i \> _{i \in \Delta}$ and $\< B_j \> _{j \in
\Lambda}$ are {\it exactly isomorphic} if there is a bijection $\pi:
\Delta \to \Lambda$ such that the linear map from $ \< A_i \> _{i
\in \Delta}$ to $\< B_j \> _{j \in \Lambda}$ induced by $A_i \mapsto
B_{\pi(i)}$ is an algebra isomorphism.
\medskip

We now formulate our main findings. Let $d$ be a positive integer,
and $n_1, n_2, \dots, n_d$ positive integers greater than or equal
to $2$. Let $K_n$ denote the one-class association scheme of order
$n$. Let $\x=(X,\{R_i\}_{i\in[d]}) = K_{n_1} \wr K_{n_2} \wr \cdots
\wr~K_{n_d}$. Let  $\{A_0, A_1, \dots, A_d\}$ be the basis of
Bose-Mesner algebra of $\x$, and $\{E_0^*, E_1^*, \dots, E_d^*\}$
the basis of the dual Bose-Mesner algebra of $\x$ with respect to a
fixed $x\in X$. Let $v=|X| = \prod\limits_{i=1}^d n_i$, and $J =
J_v$. The following theorem characterizes the Bose-Mesner algebra of
the wreath product of one-class association schemes.

\begin{thm}
\label{thm-bose}  The association scheme $\x=(X,\{R_i\}_{i\in[d]}) =
K_{n_1} \wr K_{n_2} \wr \cdots \wr~K_{n_d}$ has the following
properties.
\begin{enumerate} \item[$(\rm{i})$] The valencies of $\x$ are  \[k_1 = n_1 - 1,\]
\[k_i =(k_0+k_1+\cdots +k_{i-1})(n_i-1)= n_1 \cdots n_{i-1} (n_i -
1),\ \mbox{for } i =2, 3, ..., d.\]
\item[$(\rm{ii})$]
\[A_i A_j = k_i A_j, \quad 0 \le i < j \le d.\]
\item[$(\rm{iii})$] \[(A_i)^2 = k_i \left ( A_0 + A_1 + \cdots + A_{i-1} +
   \frac{n_i - 2}{n_i -1 } A_i\right ), \quad 1 \le i \le d.\]
\item[$(\rm{iv})$] For any $1 \le i \le d$, $\{R_0, R_1, ..., R_i\}$ is a
closed subset; and so, $\langle A_j\rangle_{j\in [i]}$ forms a
Bose-Mesner subalgebra of $\mc{A}$.\end{enumerate}
\end{thm}
Parts (i) and (iv) of Theorem \ref{thm-bose} were first observed by
G. Bhattacharyya in her Ph.D. dissertation \cite{Bh}. We derive the
following structure theorems for the Terwilliger algebra $\t(x)$ of
$\x$.

\begin{thm}
\label{thm-ter} Let $\t(x)$ be the Terwilliger algebra of $\x =
K_{n_1} \wr K_{n_2} \wr \cdots \wr~K_{n_d}$, where $n_i \ge 2$ for
all $1 \le i \le d$. For any $i, j \in [d]$, let
$$
G_{ij} := \left \{
\begin{array}{ll}
k_j^{-1} E_i^* A_j E_j^*, & \hbox{if } i < j;  \\
k_j^{-1} E_i^* A_i E_j^*, & \hbox{if } i > j; \hrow \\
k_i^{-1} E_i^* J E_i^*, & \hbox{if } i = j. \hrow
\end{array}
\right.
$$
Let $\u$ be the $\c$-space spanned by the set
$\{ G_{ij} \}_{i,j\in [d]}$. Then the following hold.
\begin{enumerate}
\item[$(\rm{i})$] $\u$ is an algebra and isomorphic to
 $M_{d+1}(\c)$.
\item[$(\rm{ii})$] $\u$ is an ideal of $\t(x)$ and the quotient algebra
$\t(x)/\u$ is commutative.
\end{enumerate}
\end{thm}

\begin{cor}\label{cor-thm-ter}\cite[Corollary 4.3]{BS} If $b$ denotes the number
$|\{i\in\{1, 2,\dots, d\} :\ n_i = 2 \}|$, then
$$
\t(x) \cong M_{d+1}(\c) \oplus M_1(\c)^{\oplus \frac{d(d+1)}{2}
-b}.$$\end{cor}

The ideal $\mc{U}$ in the above theorem is the primary ideal of
$\t(x)$ related to the primary module in \cite{BS, Ta}. Each of the
$d(d+1)/2 -b$ non-primary ideals is one-dimensional and
spanned by a central idempotent.
All of these non-primary ideals of $\t(x)$ are described in the next
theorem.

\begin{thm}
\label{thm-ter-1} Let $\t(x)$ be the Terwilliger algebra of $\x =
K_{n_1} \wr K_{n_2} \wr \cdots \wr~K_{n_d}$, where $n_i \ge 2$ for
all $1 \le i \le d$.  For any $i \in \{1, 2, ..., d\}$ and any $h
\in [i-1]$, let
$$
F_{ih} = \left \{
\begin{array}{ll}
\frac{\sum_{j=0}^h E_i^* A_j E_i^*} {\sum_{j=0}^h k_j} -
\frac{\sum_{j=0}^{h+1} E_i^* A_j E_i^*} {\sum_{j=0}^{h+1} k_j},
 & \hbox{if } h < i-1; \\
\frac{\sum_{j=0}^{i-1} E_i^* A_j E_i^*} {\sum_{j=0}^{i-1} k_j} -
G_{ii}, & \hbox{if } h = i-1; \hhrow
\end{array}
\right.
$$Then the set
$$
\{ F_{ih} :\ i \in \{1, 2, ..., d\},\ h \in [i-1] \}
$$
has $d(d+1)/2 -b$ nonzero elements, where $b = |\{i\in \{1, 2,
\dots, d\} :\ n_i = 2 \}|$, and each nonzero element is a central
idempotent that spans a $1$-dimensional non-primary ideal of
$\t(x)$.
\end{thm}

In the following section, we will study the structure of the
Bose-Mesner algebra of the wreath product and derive the properties
that characterize the wreath product of one-class association
schemes. We shall see that Theorem \ref{thm-bose} is directly
deduced from Lemma \ref{lem-wreath} by induction on $d$. We will
prove the theorems \ref{thm-ter} and \ref{thm-ter-1} in Section
\ref{sect-ter}. We shall see that our derivation of all these
results is chiefly based on a set of equations in adjacency matrices
of the association scheme.

\section{The Bose-Mesner algebra of $K_{n_1} \wr K_{n_2} \wr \cdots
\wr~K_{n_d}$}

In this section, we give a description of the Bose-Mesner algebras
of wreath products of one-class \as s that will be used in the
subsequent section. Let $K_n$ denote the one-class association
scheme of order $n$; so, its unique nontrivial relation graph is the
complete graph on $n$ vertices.

\begin{lemma}
\label{lem-wreath} Let $\x =$ \xr\ be an \as. Suppose $\x = \y \wr
K_n$ for an \as\  $\y = $ \ys. Then $e = d-1$, and by renumbering
$R_1, R_2, \dots, R_d$ if necessary, the following hold.
\begin{enumerate}
\item[$(\rm{i})$] The set $\{R_i \}_{i\in [d-1]}$ is a closed subset of $\x$
such that the Bose-Mesner subalgebra $\langle
A_i\rangle_{i\in[d-1]}$ is exactly isomorphic to the Bose-Mesner
algebra of $\y$.
\item[$(\rm{ii})$] For each $i\in \{1, 2, \dots, d\}$,
\begin{equation}
\label{eq-aiad} A_i A_d = k_i A_d,
\end{equation}
and
\begin{equation}
\label{eq-adsquare} (A_d)^2 =
\left(\sum\limits_{j=0}^{d-1}k_j\right)
\left\{(n-1)\left(\sum\limits_{j=0}^{d-1}A_j\right) + (n-2)A_d
\right\}.
\end{equation} \end{enumerate}
\end{lemma}

\begin{proof}[Proof]
Let $B_j$ be the adjacency matrix of $S_j$, for $j \in [e]$, and
$\bar{J}_n = J_n - I_n$. Then the adjacency matrices of $\y \wr K_n$
are
$$
I_n \otimes B_0, \ I_n \otimes B_1,  \cdots, I_n \otimes B_e, \
\bar{J}_n\otimes J_u,
$$
where $u = |Y|$. Thus, $e = d-1$. Let the valency of $S_j$ be
$\tilde{k_j}$. Then clearly
$$
(I_n \otimes B_j) (\bar{J}_n\otimes J_u) = \tilde{k_j}
(\bar{J}_n\otimes J_u), \ \ 1 \le j \le e,
$$
and
\begin{eqnarray*}
(\bar{J}_n\otimes J_u)^2  & = & u \Big( (n-1) I_n \otimes J_u +
     (n-2) \bar{J}_n\otimes J_u \Big) \\
 & = & u \Big( (n-1) \sum_{j=0}^e (I_n \otimes B_j) +
(n-2) (\bar{J}_n\otimes J_u) \Big ).
\end{eqnarray*}
By renumbering $R_1, R_2, \dots, R_d$ if necessary, we may assume
that the adjacency matrix of $R_i$ is $A_i = I_n \otimes B_i$, $1
\le i \le d-1$, and the adjacency matrix of $R_d$ is $A_d =
\bar{J}_n\otimes J_u$. Then the valency of $R_i$ is $\tilde{k_i}$,
$1 \le i \le d-1$, and (\ref{eq-aiad}), (\ref{eq-adsquare}) holds.
Furthermore, clearly $\{R_1, R_2, \dots, R_{d-1}\}$ is a closed
subset, and the Bose-Mesner subalgebra $\langle A_i\rangle_{i\in
[d-1]}$ is exactly isomorphic to the Bose-Mesner algebra of $\y$.
\end{proof}

As a consequence, we have Theorem \ref{thm-bose}, from which we can
easily prove the following proposition. We will need this
proposition in Section \ref{sect-ter}.

\begin{prop}
\label{prop-bose} Let $d$ be a positive integer, and let $n_1, n_2,
\dots, n_d$ be positive integers greater than or equal to $2$. Let
$\x=(X,\{R_i\}_{i\in[d]}) = K_{n_1} \wr K_{n_2} \wr \cdots \wr
K_{n_d}$, and let $A_0, A_1, \dots, A_d$ be the adjacency matrices
of $\x$. Then the following hold.
\begin{enumerate}
\item[$(\rm{i})$] For any $h \in [d]$,
$$
\left( \sum_{i=0}^h A_i \right)^2 = \Big( \sum_{i=0}^h k_i \Big)
\Big( \sum_{i=0}^h A_i \Big).
$$
\item[$(\rm{ii})$] For any $g, h \in [d]$ such that $g \le h$,
$$
A_g \Big( \sum_{i=0}^h A_i \Big) =
 k_g \Big( \sum_{i=0}^h A_i \Big).
$$\end{enumerate}
\end{prop}

\begin{proof}[Proof]
$(\rm{i})$ We use induction on $h$. Clearly $(\rm{i})$ holds when $h
= 0$. Assume that $h > 0$ and $(\rm{i})$ holds for $h-1$. Then we
show that $(\rm{i})$ holds for $h$. Recall that for any $g \in [d]$,
$k_0 + k_1 + \cdots +k_g = n_1 n_2 \cdots n_g$ and $k_g = (k_0 + k_1
+ \cdots +k_{g-1})(n_g - 1)$; and so, by Theorem \ref{thm-bose},
$$
(A_h)^2 = k_h \Big ( \sum_{i=0}^{h-1} A_i \Big) + \Big (
\sum_{i=0}^{h-1} k_i \Big)(n_h -2) A_h.
$$
Thus, the induction hypothesis and Theorem \ref{thm-bose} yield that
\begin{eqnarray*}
\left( \sum_{i=0}^h A_i \right)^2 & = &
\left( \sum_{i=0}^{h-1} A_i \right)^2 +
2 \Big ( \sum_{i=0}^{h-1} A_i \Big) A_h + (A_h)^2 \\
 & = & \Big ( \sum_{i=0}^{h-1} k_i \Big)
      \Big ( \sum_{i=0}^{h-1} A_i \Big) +
   2 \Big ( \sum_{i=0}^{h-1} k_i \Big) A_h
  + k_h \Big ( \sum_{i=0}^{h-1} A_i \Big) +
\Big ( \sum_{i=0}^{h-1} k_i \Big)(n_h -2) A_h \\
& = & \Big ( \sum_{i=0}^h k_i \Big)
      \Big ( \sum_{i=0}^h A_i \Big).
\end{eqnarray*}

$(\rm{ii})$ Since $g \le h$, by Theorem \ref{thm-bose} we see that
$$
A_g \Big( \sum_{i=0}^h A_i \Big) = A_g \Big( J - \sum_{i= h+1}^d A_i
\Big) = A_g J - \sum_{i= h+1}^d A_g A_i = k_g J - \sum_{i= h+1}^d
k_g A_i = k_g \Big( \sum_{i=0}^h A_i \Big).
$$
\end{proof}

The next proposition says that the equations
in Theorem \ref{thm-bose}(ii) characterize the Bose-Mesner algebra
of the wreath product of one-class \as s.

\begin{prop}
Let $\x=(X,\{R_i\}_{i\in[d]})$ be a commutative \as, and
$A_0, A_1$, \dots, $A_d$  the adjacency matrices of $\x$. Assume
that
$$
A_i A_j = k_i A_j, \ 0 \le i < j \le d.
$$
Then the Bose-Mesner algebra of $\x$ is exactly isomorphic to the
Bose-Mesner algebra of the wreath product of some one-class \as s.
\end{prop}

\begin{proof}[Proof]
It follows from $A_i A_j = k_i A_j$, $0 \le i < j \le d$, that
$k_0, k_1, ..., k_{d-1}$ are the valencies of $\x$, and
for any $1 \le j \le d$, we have that
$$
\left ( \sum_{i=0}^{j-1} A_i \right ) A_j =
\sum_{i=0}^{j-1} A_i A_j =
\left ( \sum_{i=0}^{j-1} k_i \right ) A_j.
$$
But on the other hand,
$$
\left ( \sum_{i=0}^{j-1} A_i \right ) A_j =
\left ( J - \sum_{i=j}^{d} A_i \right ) A_j
= k_j J - (A_j)^2 - k_j \sum_{i=j+1}^{d} A_i
= k_j \sum_{i=0}^{j} A_i - (A_j)^2.
$$
Thus,
$$
(A_j)^2 = k_j \sum_{i=0}^{j-1} A_i +
\left( k_j - \sum_{i=0}^{j-1} k_i \right ) A_j.
$$
Therefore,
for every $i\in [d]$, $\{R_h\}_{h\in [i]}$ is a closed subset of
$\x$.  In particular,  for every $1 \le i \le d$, the subset
$\{R_h\}_{h\in [i-1]}$ is also a
closed subset of $\{R_h\}_{h\in [i]}$. So the valency of the subset
$\{R_h\}_{h\in [i-1]}$ divides the valency of the subset
$\{R_h\}_{h\in [i]}$.  That is, $k_0 + k_1 + \cdots + k_{i-1}$
divides $k_0 + k_1 + \cdots + k_i$. ($k_d$ is the valency of $R_d$.)
Hence, $k_0 + k_1 + \cdots + k_{i-1}$ divides $k_i$. Let
$$
n_i = \frac{k_i}{k_0 + k_1 + \cdots + k_{i-1}} + 1, \
 1 \le i \le d.
$$
Then each $n_i$ is a positive integer greater than or equal to $2$.
Clearly $1 + k_1 = n_1$, and for any $2 \le j \le d$,
$$
1 + k_1 + \cdots + k_j = n_1 \cdots n_j  \quad \hbox{and} \quad
k_j = n_1 \cdots n_{j-1} (n_j - 1).
$$
Hence,
$$
k_j - \sum_{i=0}^{j-1} k_i = \frac{k_j(n_j - 2)}{n_j - 1}, \
 1 \le j \le d.
$$
Therefore, by Theorem \ref{thm-bose} we see that the Bose-Mesner
algebra of $\x$ is exactly isomorphic to the Bose-Mesner algebra of
the \as\ $K_{n_1}\wr K_{n_2}\wr \cdots \wr K_{n_d}$.
\end{proof}

\section{The Terwilliger algebra of $K_{n_1} \wr K_{n_2} \wr \cdots
\wr~K_{n_d}$  \label{sect-ter} }
\setcounter{equation}{0}
The purpose of this section is to prove our main results stated in
Theorems \ref{thm-ter} and \ref{thm-ter-1} as well as Corollary
\ref{cor-thm-ter}. In order to prove them, we shall derive several
technical lemmas first. These lemmas also help us to understand the
Terwilliger algebra of the wreath product.

Throughout the section, let $\x=(X,\{R_i\}_{i\in[d]}) = K_{n_1} \wr
K_{n_2} \wr \cdots \wr~K_{n_d}$, where $d$, $n_1$, $n_2$, ..., $n_d$
are positive integers greater than or equal to $2$. Let $v=|X| = n_1
n_2 \cdots n_d$. Let $\{A_i\}_{i\in [d]}$ be the basis of the
Bose-Mesner algebra of $\x$, and let $\{E_i^*\}_{i\in [d]}$
be the basis of the dual Bose-Mesner
algebra of $\mc{X}$ with respect to an arbitrary (fixed) $x\in X$.
 Let $J = J_v$ be the $v \times v$ all-ones matrix.

The next two lemmas will be used heavily in the sequel. They are
found in \cite{BS}, but the proofs here are different from those in
\cite{BS}. We show that these results are also derived from
Proposition \ref{tripro} and Theorem \ref{thm-bose}.

\begin{lemma}\cite[Theorem 3.5]{BS}
\label{lem-nonzero} For any $i, j, h \in [d]$, the following hold.

$(\rm{i})$ If $i \ne j$, then $E_i^* A_j E_h^* \ne 0$ if and only if
$h = \max\{i, j\}$.

$(\rm{ii})$ If $i < j$, then $E_i^* A_h E_j^* \ne 0$ if and only if
$h = j$.

$(\rm{iii})$ If $h > i$, then $E_i^* A_i E_h^* = 0$.

$(\rm{iv})$ If $h < j$, then $E_i^* A_h E_j^* \ne 0$ if and only if
$i = j$.
\end{lemma}

\begin{proof}[Proof]
$(\rm{i})$ Since $i \ne j$, $p_{ij}^h \ne 0$ if and only if $h =
\max\{i, j\}$ by Theorem \ref{thm-bose}. So $E_i^* A_j E_h^* \ne 0$
if and only if $h = \max\{i, j\}$ by Proposition \ref{tripro}.

Similarly, $(\rm{ii})$ and $(\rm{iii})$ follow from Theorem
\ref{thm-bose} and Proposition \ref{tripro}, and $(\rm{iv})$ follows
from (i) and (iii).
%
%
\end{proof}

\begin{lemma}\cite[Lemma 3.7]{BS}
\label{lem-ter-basic} For any $i, j, h \in [d]$, the following hold.

$(\rm{i})$ If $i < j$, then $E_i^* A_j E_j^* = E_i^* A_j$, $E_j^*
A_j E_i^* = A_j E_i^*$, and $E_j^* A_i E_j^* = A_i E_j^*$.

$(\rm{ii})$ If $i < j$, then $E_i^* A_j E_j^* = E_i^* J E_j^*$.

$(\rm{iii})$ $E_i^* J E_i^* = E_i^* (A_0 + A_1 + \cdots + A_i)
E_i^*$.

$(\rm{iv})$ $A_i E_i^* = (E_0^* + E_1^* + \cdots + E_i^*) A_i E_i^*$
and $E_i^* A_i = E_i^* A_i (E_0^* + E_1^* + \cdots + E_i^*)$.

\end{lemma}

\begin{proof}[Proof]
$(\rm{i})$ Since $i < j$, Lemma \ref{lem-nonzero}$(\rm{i})$ implies
that
$$
E_i^* A_j = E_i^* A_j (E_0^* + E_1^* + \cdots + E_d^*) =
E_i^* A_j E_j^*
$$
and  $A_j E_i^* = (E_0^* + E_1^* + \cdots + E_d^*) A_j E_i^* = E_j^*
A_j E_i^*$.\quad Also Lemma \ref{lem-nonzero}$(\rm{iv})$ yields that
$$
A_i E_j^* = (E_0^* + E_1^* + \cdots E_d^*) A_i E_j^* =
   E_j^* A_i E_j^*.
$$

$(\rm{ii})$ Since $i < j$, Lemma \ref{lem-nonzero}$(\rm{ii})$ yields
that
$$
E_i^* J E_j^* =  E_i^* (A_0 + A_1 + \cdots + A_d) E_j^*
  = E_i^* A_j E_j^*.
$$

Similarly, $(\rm{iii})$ and $(\rm{iv})$ follow from Lemma
\ref{lem-nonzero}$(\rm{i})$.
%
%
\end{proof}

For the rest of this section, given a fixed $x=x_{01}\in X$, the
rows and columns of any $E_i^*$ and $A_j$, for $i, j \in [d]$, are
indexed by the elements of $X$ in the following order
$$
x_{01}, x_{11}, ..., x_{1 k_1}, x_{21}, ..., x_{2 k_2}, ..., x_{d1},
..., x_{d k_d}
$$
such that $(x, x_{hl}) \in R_h$, $h \in [d]$ and $1 \le l \le k_h$.
We will write any $E_i^*$ and $A_j$, for $i, j \in [d]$, as block
matrices such that
$$
E_i^* = \begin{pmatrix}
0 \cr
 & \ddots \cr
 & & 0 \cr
 & & & I_{k_i} \cr
 & & & & 0 \cr
 & & & & & \ddots \cr
 & & & & & & 0
\end{pmatrix}
\quad \hbox{and} \quad
A_j = \begin{pmatrix}
(A_j)_{00} & (A_j)_{01} &(A_j)_{02} & \cdots & (A_j)_{0d} \cr
(A_j)_{10} & (A_j)_{11} &(A_j)_{12} & \cdots & (A_j)_{1d} \cr
(A_j)_{20} & (A_j)_{21} &(A_j)_{22} & \cdots & (A_j)_{2d} \cr
\vdots     & \vdots     & \vdots    & \ddots & \vdots  \cr
(A_j)_{d0} & (A_j)_{d1} &(A_j)_{d2} & \cdots & (A_j)_{dd}
\end{pmatrix},
$$
where $(A_j)_{rs}$ is a $k_r \times k_s$ matrix, called the $(r,
s)$-{\it block} of $A_j$, for $r, s \in [d]$. We will also write any
matrix in the Terwilliger algebra $\t(x)$ as a block matrix in the
same way. In particular, for any $i, j, h \in [d]$, we write $E_i^*
A_j E_h^*$ as a block matrix. Thus, the $(i, h)$-block of $E_i^* A_j
E_h^*$ is $(A_j)_{ih}$, and any other block of $E_i^* A_j E_h^*$ is
zero. We remark that all the results in this paper except for Lemmas
\ref{lem-aj} and \ref{lem-eij} below are independent of the choice
of the order of elements in $X$.

For any positive integers $p$ and $q$, let $J_{p, q}$ be the $p
\times q$ matrix whose entries are all $1$. The next lemma plays an
important role in our discussion.

\begin{lemma}
\label{lem-aj} For any $j \in [d]$,
$$
A_j = \begin{pmatrix}
 0 & 0 & \cdots & 0 & J_{1, k_j} & 0 & \cdots & 0 \cr
 0 & 0 & \cdots & 0 & J_{k_1, k_j} & 0 & \cdots & 0 \cr
\vdots & \vdots & & \vdots & \vdots & \vdots & & \vdots \cr
 0 & 0 & \cdots & 0 & J_{k_{j-1}, k_j} & 0 & \cdots & 0 \cr
J_{k_j, 1} & J_{k_j, k_1} & \cdots & J_{k_j, k_{j-1}} &
  (A_j)_{jj} & 0 & \cdots & 0 \cr
 0 & 0 & \cdots & 0 & 0 & (A_j)_{j+1, j+1} & \cdots & 0 \cr
\vdots & \vdots & & \vdots & \vdots & \vdots & \ddots & \vdots \cr
 0 & 0 & \cdots & 0 & 0 & 0 & \cdots & (A_j)_{dd}
\end{pmatrix}.
$$
\end{lemma}

\begin{proof}[Proof]
Let $i, h \in [d]$. If $i < j$, then $E_i^* A_j E_h^* \ne 0$ if and
only if $h = j$ by Lemma \ref{lem-nonzero}$(\rm{i})$, and hence
$(A_j)_{ih} \ne 0$ if and only if $h = j$. If $i > j$, then $E_i^*
A_j E_h^* \ne 0$ if and only if $h = i$ by Lemma
\ref{lem-nonzero}$(\rm{i})$, and hence $(A_j)_{ih} \ne 0$ if and
only if $h = i$. Moreover, since $E_j^* A_j E_h^* = 0$ for any $h >
j$ by Lemma \ref{lem-nonzero}$(\rm{iii})$, we see that $(A_j)_{jh} =
0$ for any $h > j$. Thus, $A_j$ has the form
$$
A_j = \begin{pmatrix}
 0 & 0 & \cdots & 0 & (A_j)_{0j} & 0 & \cdots & 0 \cr
 0 & 0 & \cdots & 0 & (A_j)_{1j} & 0 & \cdots & 0 \cr
\vdots & \vdots & & \vdots & \vdots & \vdots & & \vdots \cr
 0 & 0 & \cdots & 0 & (A_j)_{j-1,j} & 0 & \cdots & 0 \cr
(A_j)_{j0} & (A_j)_{j1} & \cdots & (A_j)_{j,j-1} &
  (A_j)_{jj} & 0 & \cdots & 0 \cr
 0 & 0 & \cdots & 0 & 0 & (A_j)_{j+1, j+1} & \cdots & 0 \cr
\vdots & \vdots & & \vdots & \vdots & \vdots & \ddots & \vdots \cr
 0 & 0 & \cdots & 0 & 0 & 0 & \cdots & (A_j)_{dd}
\end{pmatrix}.
$$
For any $i < j$, since $E_i^* A_j E_j^* = E_i^* J E_j^*$ by Lemma
\ref{lem-ter-basic}$(\rm{ii})$, we see that $(A_j)_{ij} = J_{k_i,
k_j}$. Also we have $(A_j)_{ji} = J_{k_j, k_i}$ for any $i < j$
since $A_j$ is symmetric. So the lemma holds.
\end{proof}

Motivated by this lemma, we define the following matrices $G_{ij}$,
for all $i, j \in [d]$. Let
$$
G_{ij} := \left \{
\begin{array}{ll}
k_j^{-1} E_i^* A_j E_j^*, & \hbox{if } i < j;  \\
k_j^{-1} E_i^* A_i E_j^*, & \hbox{if } i > j; \hrow \\
k_i^{-1} E_i^* J E_i^*, & \hbox{if } i = j. \hrow
\end{array}
\right.
$$
Clearly, $\{ G_{ij} :\  i, j\in [d] \}$ is a linearly independent
subset of $\t(x)$. Let $\u$ be the $\c$-space with basis $\{ G_{ij}
:\ i, j \in [d] \}$. Thus, the dimension of $\u$ is $(d+1)^2$.

The next lemma is a direct consequence of Lemma \ref{lem-aj}.

\begin{lemma}
\label{lem-eij} For any $i, j \in [d]$, the $(i, j)$-block of
$G_{ij}$ is $k_j^{-1} J_{k_i, k_j}$, and any other block of $G_{ij}$
is zero.
\end{lemma}

Theorem \ref{thm-ter}$(\rm{i})$ follows from the next lemma.

\begin{lemma}
\label{lem-eijekl} For any $i, j, g, h \in [d]$,
\begin{equation}
\label{eq-eijekl} G_{ij} G_{gh} = \delta_{jg} G_{ih},
\end{equation}
where $\delta_{jg}$ is the Kronecker delta.
\end{lemma}

\begin{proof}[Proof]
 Clearly
$G_{ij} G_{gh} = 0$ if $j \ne g$. So we only need to show that
$$
G_{ij} G_{jh} = G_{ih}.
$$
Since the $(i, j)$-block is the only nonzero block of $G_{ij}$, and
the $(j, h)$-block is the only nonzero block of $G_{jh}$, we see
that every block
 except for the $(i, h)$-block of $G_{ij} G_{jh}$ is zero.
By Lemma \ref{lem-eij},  the $(i, h)$-block of $G_{ij} G_{jh}$ is
$$
k_j^{-1} J_{k_i, k_j} \cdot k_h^{-1} J_{k_j, k_h} = k_h^{-1} J_{k_i,
k_h}.
$$
Thus, $G_{ij} G_{jh} = G_{ih}$ by Lemma \ref{lem-eij}.
\end{proof}

\begin{lemma}\label{lem-ter-basic-2} For $h, i, j\in [d]$, the
following hold.

 $(\rm{i})$ If $h < i$, then $A_h G_{ij} = E_i^* A_h E_i^*
G_{ij}$.

$(\rm{ii})$ If $h < j$, then $G_{ij} A_h = G_{ij} E_j^* A_h E_j^*$.
\end{lemma}
\begin{proof}
$(\rm{i})$ Since $h < i$, $A_h E_i^* = E_i^* A_h E_i^*$ by Lemma
\ref{lem-ter-basic}$(\rm{i})$. Thus,
$$
A_h G_{ij} = A_h E_i^* G_{ij} = E_i^* A_h E_i^* G_{ij}.
$$

The proof of $(\rm{ii})$ is similar.
\end{proof}

The first half of Theorem \ref{thm-ter}$(\rm{ii})$ will follow from
the next lemma.

\begin{lemma}
\label{lem-aheij} For any $i, j, h \in [d]$, the following hold.

$(\rm{i})$
$$
A_h G_{ij} = \left \{
\begin{array}{ll}
k_h G_{ij}, & \hbox{if } h < i; \\
k_i \Big (\sum\limits_{r=0}^{i-1} G_{rj} \Big) + \Big (k_i -
\sum\limits_{r=0}^{i-1} k_r \Big) G_{ij}, & \hbox{if } h = i;
             \hrow \\
k_i G_{hj}, & \hbox{if } h > i. \hrow
\end{array}
\right.
$$

$(\rm{ii})$
$$
 G_{ij} A_h = \left \{
\begin{array}{ll}
k_h G_{ij}, & \hbox{if } h < j; \\
\sum\limits_{r=0}^{j-1} (k_r G_{ir}) + \Big (k_j -
\sum\limits_{r=0}^{j-1} k_r \Big) G_{ij}, & \hbox{if } h = j;
             \hrow \\
k_h G_{ih}, & \hbox{if } h > j. \hrow
\end{array}
\right.
$$
\end{lemma}

\begin{proof}[Proof]
$(\rm{i})$ If $h < i$, then $A_h G_{ij} = E_i^* A_h E_i^* G_{ij}$ by
Lemma \ref{lem-ter-basic-2}$(\rm{i})$. From
 Lemma \ref{lem-aj}, the $(i, i)$-block of $E_i^* A_h E_i^*$
is $(A_h)_{ii}$, and any other block of $E_i^* A_h E_i^*$ is zero.
Note that the sum of entries in any row of $(A_h)_{ii}$ is $k_h$ by
Lemma \ref{lem-aj}. So Lemma \ref{lem-eij} yields that the $(i,
j)$-block of $(E_i^* A_h E_i^*) G_{ij}$ is
$$
(A_h)_{ii} \cdot k_j^{-1} J_{k_i, k_j} = k_h k_j^{-1} J_{k_i, k_j},
$$
and any other block of $(E_i^* A_h E_i^*) G_{ij}$ is zero. Thus,
$(E_i^* A_h E_i^*) G_{ij} = k_h G_{ij}$ by Lemma \ref{lem-eij}, and
$(\rm{i})$ holds for $h < i$.

 If $h = i$, then by Lemma \ref{lem-ter-basic}$(\rm{iv})$,
$$
A_h G_{ij} = A_i G_{ij} = A_i E_i^* G_{ij}
   = \sum_{r=0}^i E_r^* A_i E_i^* G_{ij}
  = \sum_{r=0}^{i-1} E_r^* A_i E_i^* G_{ij} +
        E_i^* A_i E_i^* G_{ij}.
$$
By Lemma \ref{lem-eijekl},
$$
\sum_{r=0}^{i-1} E_r^* A_i E_i^* G_{ij} =
 \sum_{r=0}^{i-1} k_i G_{ri} G_{ij} =
k_i \sum_{r=0}^{i-1} G_{rj}.
$$
From Lemmas \ref{lem-ter-basic}$(\rm{iii})$,
\ref{lem-ter-basic-2}$(\rm{i})$, \ref{lem-eijekl}, and what we have
just proved,
$$
E_i^* A_i E_i^* G_{ij} = \Big(E_i^* J E_i^* -
  \sum_{r=0}^{i-1} E_i^* A_r E_i^* \Big) G_{ij} =
k_i G_{ii} G_{ij} - \sum_{r=0}^{i-1} A_r G_{ij} = \Big(k_i -
\sum_{r=0}^{i-1} k_r \Big) G_{ij}.
$$
Thus,
$$
A_i G_{ij} = k_i \sum_{r=0}^{i-1} G_{rj} +
    \Big(k_i - \sum_{r=0}^{i-1} k_r \Big) G_{ij}.
$$
So $(\rm{i})$ holds for $h = i$.

If $h > i$, then Lemmas \ref{lem-ter-basic}$(\rm{i})$ and
\ref{lem-eijekl} imply that
$$
A_h G_{ij} = A_h E_i^* G_{ij} = E_h^* A_h E_i^* G_{ij} = k_i G_{hi}
G_{ij} = k_i G_{hj}.
$$
So $(\rm{i})$ holds for $h > i$.

The proof of $(\rm{ii})$ is similar.
\end{proof}

The next lemma is needed for the proof of the second half of Theorem
\ref{thm-ter}$(\rm{ii})$ and the proof of Theorem \ref{thm-ter-1}.

\begin{lemma}
\label{lem-eiahei} For any $i, g, h \in [d]$ such that $g, h \in
[i]$ and at least one of $g$ and $h$ is not equal to $i$, we have
that
\begin{equation}
\label{eq-comm} (E_i^* A_g E_i^*) (E_i^* A_h E_i^*) = E_i^* A_g A_h
E_i^*.
\end{equation}
\end{lemma}

\begin{proof}[Proof]
For each $i \in [d]$ and every $g,h \in [i]$, Lemma \ref{lem-aj}
implies that the $(i, i)$-block of $(E_i^* A_g E_i^*) (E_i^* A_h
E_i^*)$ is $(A_g)_{ii} (A_h)_{ii}$. If at least one of $g$ and $h$
is not equal to $i$, then by Lemma \ref{lem-aj}, the $(i, i)$-block
of $E_i^* A_g A_h E_i^*$ is also $(A_g)_{ii} (A_h)_{ii}$. Thus,
$(E_i^* A_g E_i^*) (E_i^* A_h E_i^*) = E_i^* A_g A_h E_i^*$.
\end{proof}

Now we are ready to prove Theorem \ref{thm-ter} and Corollary
\ref{cor-thm-ter}.

\begin{proof}[Proof of Theorem \ref{thm-ter}]
$(\rm{i})$ For any $i, j, h \in [d]$, $A_h G_{ij} \in \u$ and
$G_{ij} A_h \in \u$ by Lemma \ref{lem-aheij}, and $E_h^* G_{ij} =
\delta_{hi} G_{ij} \in \u$, $G_{ij} E_h^* = \delta_{jh} G_{ij} \in
\u$. So $\u$ is an ideal of $\t(x)$.
For any
$i, j \in [d]$, let $E_{ij}$ be the $(d+1) \times (d+1)$ matrix
whose $(i, j)$-entry is $1$ and whose other entries are all zero.
Then the linear map $\varphi: \u \to M_{d+1}(\c)$ defined by
$$
\varphi(G_{ij}) = E_{ij}, \ \  i, j \in [d],
$$
establishes an isomorphism by (\ref{eq-eijekl}). Thus, $\u$ is
isomorphic to $M_{d+1}(\c)$.

$(\rm{ii})$ Recall that $\x = K_{n_1} \wr K_{n_2} \wr \cdots
\wr~K_{n_d}$ is triply-regular, and hence $\t(x) = \t_0(x)$. Thus,
$$
\t(x)/\u = \{ E_i^* A_h E_i^* + \u :\ i\in [d], \
  \ h \in [i] \}.
$$
For any $i \in [d]$ and any $g,h\in [i]$, Lemma \ref{lem-eiahei}
implies that
$$
(E_i^* A_g E_i^*) (E_i^* A_h E_i^*) = E_i^* A_g A_h E_i^* = E_i^*
A_h A_g E_i^* = (E_i^* A_h E_i^*) (E_i^* A_g E_i^*), \quad \hbox{ if
} g \ne h.
$$
Thus, $\t(x)/\u$ is commutative.
\end{proof}

\begin{proof}[Proof of Corollary \ref{cor-thm-ter}]
Since $\t(x)$ is semisimple, and $\u$ is an ideal of $\t(x)$ by
Theorem \ref{thm-ter}$(\rm{ii})$, we see that $\t(x) = \u \oplus \i$
for some ideal $\i$ of $\t(x)$. Note that $\i$ is also semisimple by
Wedderburn-Artin's Theorem, and $\i \cong \t(x)/\u$ is a commutative
algebra by Theorem \ref{thm-ter}$(\rm{ii})$. Recall that the
dimension of $\t(x)$ is
$$
(d+1)^2 + \frac{d(d+1)}{2} - b,
$$
where $b = |\{i\in\{1, 2, \dots, d\} :\ n_i = 2\}|$, and the
dimension of $\u$ is $(d+1)^2$. So the dimension of $\i$ is
$d(d+1)/2 - b$, and
$$
\i \cong M_1(\c)^{\oplus \frac{d(d+1)}{2} - b}.
$$
This completes the proof.
\end{proof}

In order to prove Theorem \ref{thm-ter-1}, we first prove the
following two lemmas.

\begin{lemma}
\label{lem-idemp} For any $i\in \{1, 2, \dots, d\}$ and any $h \in
[i-1]$, we have that

\begin{equation}
\label{eq-idemp}
\left (\sum_{j=0}^h E_i^* A_j E_i^* \right )^2 =
 \Big (\sum_{j=0}^h k_j \Big )
\Big (\sum_{j=0}^h E_i^* A_j E_i^* \Big ).
\end{equation}
\end{lemma}

\begin{proof}[Proof]
For any $i\in\{1, 2, \dots, d\}$ and any $h \in [i-1]$, by Lemma
\ref{lem-eiahei} we have that
\begin{eqnarray*}
\left (\sum_{j=0}^h E_i^* A_j E_i^* \right )^2  & = &
  \sum_{l=0}^h \sum_{j=0}^h (E_i^* A_j E_i^*)(E_i^* A_l E_i^*) \\
 & = & \sum_{l=0}^h \sum_{j=0}^h E_i^* A_j A_l E_i^* \\
 & = & E_i^* (A_0 + A_1 + \cdots + A_h)^2 E_i^*.
\end{eqnarray*}
So (\ref{eq-idemp}) holds by Proposition \ref{prop-bose}$(\rm{i})$.

\end{proof}

\begin{lemma}
\label{lem-ak} For any $i, h, g \in [d]$ such that $i \ge 1$ and
$h\in [i-1]$, the following hold.

$(\rm{i})$
$$
A_g \Big ( \sum_{j=0}^h E_i^* A_j E_i^* \Big ) = \left \{
\begin{array}{ll}
k_g \sum_{j=0}^h E_i^* A_j E_i^*, & \hbox{if } g \le h; \\
\Big ( \sum_{j=0}^h k_j \Big ) E_i^* A_g E_i^*;
     & \hbox{if } h < g < i; \hrow \\
\Big ( \sum_{j=0}^h k_j \Big ) A_i E_i^*;
     & \hbox{if } g = i; \hrow \\
\Big ( \sum_{j=0}^h k_j \Big ) E_g^* A_g E_i^*;
 & \hbox{if } g > i. \hrow
\end{array}
\right.
$$

$(\rm{ii})$
$$
 \Big ( \sum_{j=0}^h E_i^* A_j E_i^* \Big ) A_g =
\left \{
\begin{array}{ll}
k_k \sum_{j=0}^h E_i^* A_j E_i^*, & \hbox{if } g \le h; \\
\Big ( \sum_{j=0}^h k_j \Big ) E_i^* A_g E_i^*;
     & \hbox{if } h < g < i; \hrow \\
\Big ( \sum_{j=0}^h k_j \Big ) E_i^* A_i;
     & \hbox{if } g = i; \hrow \\
\Big ( \sum_{j=0}^h k_j \Big ) E_i^* A_g E_g^*;
 & \hbox{if } g > i. \hrow
\end{array}
\right.
$$
\end{lemma}

\begin{proof}[Proof]
$(\rm{i})$ If $g \le h$, then for any $j \in [h]$, since $h < i$, by
Lemmas \ref{lem-ter-basic}$(\rm{i})$ and \ref{lem-eiahei},
$$
A_g E_i^* A_j E_i^* = (E_i^* A_g E_i^*) (E_i^* A_j E_i^*) = E_i^*
A_gA_j E_i^*.
$$
Thus, Proposition \ref{prop-bose}$(\rm{ii})$ implies that
$$
A_g \Big ( \sum_{j=0}^h E_i^* A_j E_i^* \Big )   = E_i^*  \Big (A_g
\sum_{j=0}^h A_j \Big ) E_i^*
  = E_i^*  \Big (k_g \sum_{j=0}^h A_j \Big ) E_i^*.
$$
So $(\rm{i})$ holds for $g \le h$.

If $h < g <i$, then similar to the proof of the case $g \le h$,
$$
A_g \Big ( \sum_{j=0}^h E_i^* A_j E_i^* \Big ) = E_i^*  \Big (A_g
\sum_{j=0}^h A_j \Big ) E_i^*
  = E_i^*  \Big (\sum_{j=0}^h k_j \Big ) A_g E_i^*.
$$
So $(\rm{i})$ holds for $h < g < i$.

If $g = i$, then for any $l \in [i-1]$ and any $j \in [h]$, since $h
< i$, we have that
$$
E_l^* A_i E_i^* A_j E_i^* = k_i G_{li} A_j = k_j k_i G_{li} = k_j
E_l^* A_i E_i^*
$$
by Lemmas \ref{lem-ter-basic-2}$(\rm{ii})$ and
\ref{lem-aheij}$(\rm{ii})$, and
$$
E_i^* A_i E_i^* A_j E_i^* = E_i^* A_i A_j E_i^* = k_j E_i^* A_i
E_i^*
$$
 by (\ref{eq-comm}). Thus, Lemma \ref{lem-ter-basic}$(\rm{iv})$ yields
that for any $j \in [h]$,
$$
A_i E_i^* A_j E_i^* = \sum_{l=0}^i E_l^* A_i E_i^* A_j E_i^* = k_j
\sum_{l=0}^i E_l^* A_i E_i^* = k_j A_i E_i^*.
$$
Therefore,
$$
A_i \Big ( \sum_{j=0}^h E_i^* A_j E_i^* \Big ) = \sum_{j=0}^h A_i
E_i^* A_j E_i^* = \Big( \sum_{j=0}^h k_j \Big) A_i E_i^*.
$$
So $(\rm{i})$ holds for $g = i$.

If $g > i$, then for any $j \in [h]$, by Lemmas
\ref{lem-ter-basic}$(\rm{i})$, \ref{lem-ter-basic-2}$(\rm{ii})$, and
\ref{lem-aheij}$(\rm{ii})$,
$$
A_g E_i^* A_j E_i^* = E_g^* A_g E_i^* A_j E_i^* = k_i G_{ki} A_j =
k_j k_i G_{ki} = k_j E_g^* A_g E_i^*.
$$
Thus,
$$
A_g\Big ( \sum_{j=0}^h E_i^* A_j E_i^* \Big ) = \sum_{j=0}^h A_g
E_i^* A_j E_i^* = \Big( \sum_{j=0}^h k_j \Big) E_g^* A_g E_i^*.
$$
So $(\rm{i})$ holds for $g > i$.

The proof of $(\rm{ii})$ is similar.
\end{proof}

Now we are ready to prove Theorem \ref{thm-ter-1}, which claims that
each nonzero element of the set
$$
\{ F_{ih} :\ i \in \{1, 2, ..., d\},\ h \in [i-1] \}
$$
where
$$
F_{ih} = \left \{
\begin{array}{ll}
\frac{\sum_{j=0}^h E_i^* A_j E_i^*} {\sum_{j=0}^h k_j} -
\frac{\sum_{j=0}^{h+1} E_i^* A_j E_i^*} {\sum_{j=0}^{h+1} k_j},
 & \hbox{if } h < i-1; \\
\frac{\sum_{j=0}^{i-1} E_i^* A_j E_i^*} {\sum_{j=0}^{i-1} k_j} -
G_{ii}, & \hbox{if } h = i-1; \hhrow
\end{array}
\right.
$$ is a central idempotent that spans a $1$-dimensional ideal of
$\t(x)$.

\begin{proof}[Proof of Theorem \ref{thm-ter-1}]
Let $i \in \{1, 2, ..., d\}$ and $ h \in [i-1]$. Clearly $F_{ih} \ne
0$ if $h < i-1$. From Theorem \ref{thm-bose} and Proposition
\ref{tripro} we see that $E_i^* A_i E_i^* = 0$ if and only if $n_i =
2$ if and only if $k_i = k_0 + k_1 + \cdots + k_{i-1}$. Hence,
$F_{i, i-1} = 0$ if and only if $n_i = 2$. Therefore, the set
$\{F_{ih} :\   i \in\{1,2,\dots, d\},\  h \in [i-1]\}$ has $d(d+1)/2
-b$ nonzero elements, where $b = |\{i\in\{1, 2, \dots, d\} :\  n_i =
2\}|$.

In the following we show that $F_{ih}^2 = F_{ih}$, for all $i \in
\{1, 2, \dots, d\}$ and $h \in [i-1]$. First if we assume that $h <
i-1$, then by Lemmas \ref{lem-ter-basic}$(\rm{i})$ and
\ref{lem-ak}$(\rm{ii})$,
$$
\Big ( \sum_{j=0}^h E_i^* A_j E_i^* \Big ) E_i^* A_{h+1} E_i^* =
\Big ( \sum_{j=0}^h E_i^* A_j E_i^* \Big ) A_{h+1} = \Big (
\sum_{j=0}^h k_j \Big ) E_i^* A_{h+1} E_i^*.
$$
Thus, Lemma \ref{lem-idemp} implies that
\begin{eqnarray*}
\Big ( \sum_{j=0}^h E_i^* A_j E_i^* \Big )
\Big ( \sum_{j=0}^{h+1} E_i^* A_j E_i^* \Big )
 & = & \left( \sum_{j=0}^h E_i^* A_j E_i^* \right)^2 +
\Big ( \sum_{j=0}^h E_i^* A_j E_i^* \Big ) E_i^* A_{h+1} E_i^* \\
& = & \Big ( \sum_{j=0}^h k_j \Big ) \Big ( \sum_{j=0}^{h+1} E_i^*
A_j E_i^* \Big ).
\end{eqnarray*}
Similarly,
$$
\Big ( \sum_{j=0}^{h+1} E_i^* A_j E_i^* \Big ) \Big ( \sum_{j=0}^h
E_i^* A_j E_i^* \Big ) = \Big ( \sum_{j=0}^h k_j \Big ) \Big (
\sum_{j=0}^{h+1} E_i^* A_j E_i^* \Big ).
$$
Therefore, by Lemma \ref{lem-idemp} we see that
$$
F_{ih}^2 = F_{ih}, \quad \hbox{for any } i\in\{1, 2, \dots, d\}
\hbox{ and }  h\in [i-2].
$$
Now assume that $h = i-1$. By Lemma \ref{lem-aheij}$(\rm{ii})$,
$$
\Big ( \sum_{j=0}^{i-1} E_i^* A_j E_i^* \Big ) G_{ii} =
\sum_{j=0}^{i-1} E_i^* A_j G_{ii} = \Big ( \sum_{j=0}^{i-1} k_j \Big
) G_{ii}.
$$
Similarly,
$$
G_{ii} \Big ( \sum_{j=0}^{i-1} E_i^* A_j E_i^* \Big ) = \Big (
\sum_{j=0}^{i-1} k_j \Big ) G_{ii}.
$$
Thus, Lemmas \ref{lem-idemp} and \ref{lem-eijekl} yield that
$$
F_{i, i-1}^2 = F_{i, i-1}, \quad \hbox{for any } 1 \le i \le d.
$$
Therefore, each nonzero element in the set $\{F_{ih} :\  i
\in\{1,2,\dots, d\},\  h \in [i-1]\}$ is an idempotent.

Now we prove that each nonzero element in the set $\{F_{ih} :\  i
\in\{1,2,\dots, d\},\  h \in [i-1]\}$ is a central idempotent that
spans a $1$-dimensional ideal of the Terwilliger algebra $\t(x)$.
Clearly this is a direct consequence of the following
\begin{equation}
\label{eq-akfih} A_g F_{ih} = F_{ih} A_g = \left \{
\begin{array}{ll}
k_g F_{ih}, & \hbox{if } g \le h; \\
- \Big ( \sum_{j=0}^h k_j \Big ) F_{ih}, & \hbox{if } g = h+1;
          \hrow \\
0, & \hbox{if } g > h + 1. \hrow
\end{array}
\right.
\end{equation}
From Lemmas \ref{lem-aheij} and \ref{lem-ak}, we see that
(\ref{eq-akfih}) holds if $g \le h$ or $g > h+1$. Now we show that
(\ref{eq-akfih}) is true for $g = h+1$. If $h < i-1$, then by Lemma
\ref{lem-ak}$(\rm{i})$,
\begin{eqnarray*}
A_{h+1} F_{ih} & = & E_i^* A_{h+1} E_i^* -
 \frac{ k_{h+1} \sum_{j=0}^{h+1} E_i^* A_j E_i^*}
      {\sum_{j=0}^{h+1} k_j} \\
 & = & \frac{ \Big( \sum_{j=0}^h k_j \Big)
       \sum_{j=0}^{h+1} E_i^* A_j E_i^* -
      \Big( \sum_{j=0}^{h+1} k_j \Big)
       \sum_{j=0}^h E_i^* A_j E_i^* }
      {\sum_{j=0}^{h+1} k_j}  \hhrow \\
 & = & - \Big( \sum_{j=0}^h k_j \Big) F_{ih}.
\end{eqnarray*}
Similarly,
$$
F_{ih} A_{h+1} = - \Big( \sum_{j=0}^h k_j \Big) F_{ih}.
$$
Thus, (\ref{eq-akfih}) holds if $h < i-1$ and $g = h+1$. If $h =
i-1$ and $g = h+1$, then by Lemmas \ref{lem-ak} and \ref{lem-aheij},
\begin{eqnarray*}
A_i F_{i, i-1} & = & A_i E_i^* - k_i \sum_{j=0}^{i-1} G_{ji}
     - \Big( k_i - \sum_{j=0}^{i-1} k_j \Big) G_{ii} \\
 & = & A_i E_i^* - \sum_{j=0}^{i-1} E_j^* A_i E_i^* -
     E_i^* J E_i^* +  \Big(\sum_{j=0}^{i-1} k_j \Big) G_{ii}.
\end{eqnarray*}
Since $A_i E_i^* = (E_0^* + E_1^* + \cdots + E_i^*) A_i E_i^*$ by
Lemma \ref{lem-ter-basic}$(\rm{iv})$ and $E_i^* J E_i^* = E_i^* (A_0
+ A_1 + \cdots + A_i) E_i^*$ by Lemma
\ref{lem-ter-basic}$(\rm{iii})$, we obtain that
$$
A_i E_i^* - \sum_{j=0}^{i-1} E_j^* A_i E_i^* - E_i^* J E_i^*
= - \sum_{j=0}^{i-1} E_i^* A_j E_i^*.
$$
Thus,
$$
A_i F_{i, i-1} = - \Big( \sum_{j=0}^h k_j \Big) F_{i, i-1}.
$$
Similarly,
$$
F_{i, i-1} A_i = - \Big( \sum_{j=0}^h k_j \Big) F_{i, i-1}.
$$
So (\ref{eq-akfih}) holds if $h = i-1$ and $g = h+1$. This proves
(\ref{eq-akfih}). Therefore, the theorem holds.
\end{proof}

\bigskip
\begin{center}{\textbf{Acknowledgement}}\end{center}
This research was supported in part by Vernon Wilson Endowment at
Eastern Kentucky University.

 \end{document}